\documentclass[a4paper,10pt,psamsfonts]{amsart}

\usepackage[utf8]{inputenc} % direct use of umlauts
\usepackage{amsmath}
\usepackage{amstext}
\usepackage{amsfonts}
\usepackage{amsthm}
\usepackage{amssymb}
\usepackage{stmaryrd}

\usepackage{hyperref}
\newcommand{\sect}%
{
  \setcounter{equation}{0}%
  \setcounter{figure}{0}%
  \section
}

\newcommand{\R}{\mbox{$\mathbb{R}$}}

\newcommand{\E}{\mbox{$\mathbf{E}$}}

\newcommand{\F}{\mbox{$\mathcal{F}$}}

\newcommand{\D}{\mbox{$\mathcal{D}$}}

\newcommand{\LOH}{L^p(\Omega;H)}
\newcommand{\LOR}{L^p(\Omega;\mathbb{R})}
\newcommand{\dt}{\, \mathrm{d}t}
\newcommand{\ds}{\, \mathrm{d}\sigma}
\newcommand{\dWt}{\, \mathrm{d}W(t)}
\newcommand{\dWs}{\, \mathrm{d}W(\sigma)}
\newcommand{\dd}{\, \mathrm{d}}
\newcommand{\ee}{\, \mathrm{e}}
\theoremstyle{plain}

\newtheorem{definition}{Definition}[section]

\newtheorem{theorem}[definition]{Theorem}
\newtheorem{lemma}[definition]{Lemma}

\newtheorem{corollary}[definition]{Corollary}

\newtheorem{assumption}[definition]{Assumption}

\theoremstyle{definition}

\newtheorem*{remark}{Remark}
\newtheorem*{remarks}{Remarks}

\newtheorem{example}[definition]{Example}

\begin{document}

\title[Regularity for semilinear SPDE with multiplicative noise] 
{Optimal Regularity for Semilinear\\
Stochastic Partial Differential Equations\\
with Multiplicative Noise}
\author[R.~Kruse]{Raphael Kruse$^*$}
\author[S.~Larsson]{Stig Larsson$^\dag$}
\email{rkruse@math.uni-bielefeld.de}
%\URL{http://www.math.uni-bielefeld.de/~rkruse}
\email{stig@chalmers.se}
%\address{Universit\"at Bielefeld}

% address of first author
\footnotetext[1]{Department of Mathematics, Bielefeld University, P.O. Box
100131,33501 Bielefeld, Germany, \\ 
supported by CRC 701 'Spectral Structures and Topological Methods 
in Mathematics' and DFG-IGK 1132 'Stochastics and Real World Models'.}

% address of second author
\footnotetext[2]{Department of Mathematical Sciences, Chalmers University of
Technology and University of Gothenburg, SE-412 96 Gothenburg, Sweden,
supported by the Swedish Research Council (VR) and by the Swedish Foundation
for Strategic Research (SSF) through GMMC, the Gothenburg Mathematical Modelling
Centre.}

\begin{abstract}
  This paper deals with the spatial and temporal regularity of the
  unique Hilbert space valued mild solution to a semilinear stochastic
  partial differential equation with nonlinear terms that satisfy
  global Lipschitz conditions. It is shown that the mild solution has
  the same optimal regularity properties as the stochastic
  convolution. The proof is elementary and makes use of existing
  results on the regularity of the solution, in particular, the
  H\"older continuity with a non-optimal exponent.
\end{abstract}

\maketitle

\sect{Introduction}

Consider the following semilinear stochastic partial differential equation
(SPDE)
\begin{align}
  \begin{split}
  \dd X(t) + \left[ AX(t) + F(X(t)) \right] \dt &= G(X(t)) \dWt,
  \quad \text{ for } 0 \le t \le T,\\
  X(0) &= X_0,
\end{split}
\label{eq1:SPDE}
\end{align}
where the mild solution $X$ takes values in a Hilbert space $H$. The
linear operator $-A \colon D(A) \subset H \to H$ is self-adjoint and
the generator of an analytic semigroup $E(t)=\ee^{-tA}$ on $H$. For example,
let $-A$ be 
the Laplacian with homogeneous Dirichlet boundary conditions and $H =
L^2(\D)$ for some bounded domain $\D \subset \R^d$ with smooth
boundary $\partial \D$ or a convex domain with polygonal boundary.
The nonlinear operators $F$ and $G$ are assumed to be globally
Lipschitz continuous in the appropriate sense and $W\colon [0,T]
\times \Omega \to U$ denotes a standard $Q$-Wiener process on a
probability space $(\Omega,\F,P)$ with values in some Hilbert space
$U$.

In a recent paper \cite{jentzen2010b} the authors prove the existence
of a unique mild solution $X \colon [0,T] \times \Omega \to
H$. Moreover, they show that $X$ enjoys certain spatial and temporal
regularity properties.

The spatial regularity is measured in terms of the domains $\dot{H}^r
:= D(A^{\frac{r}{2}})$, $r \ge 0$, of fractional powers of the
operator $A$. If $-A$ is the Laplacian, these domains coincide with
standard Sobolev spaces, for example, $\dot{H}^1 = H^1_0(\D)$ or
$\dot{H}^2 = H^1_0(\D) \cap H^2(\D)$ (c.f.\
\cite[Th.~6.4]{larsson2003} or \cite[Ch.~3]{thomee2006}). The
regularity in time is expressed by the H\"older exponent.

Using only Lipschitz assumptions on $F$ and $G$ the authors of
\cite{jentzen2010b} show that for every $\gamma \in [0,1)$ the
solution $X$ maps into $\dot{H}^\gamma \subset H$ and is
$\frac{\gamma}{2}$-H\"older continuous with respect to the norm $\big(
\E\big[ \| \cdot \|^p_{H} \big] \big)^{\frac{1}{p}}$, $p \in [2,\infty)$.

As we will demonstrate in this paper, it turns out that this is also true for
the border case $\gamma = 1$. The proof is based on a very careful
use of the smoothing property of the corresponding semigroup $E(t) =
\ee^{-tA}$ (see Lemma \ref{lem:1}), and on the H\"older continuity of $X$
with a suboptimal exponent (see Lemmas \ref{lem:2} and \ref{lem:3}).

The case $\gamma = 1$ is of special interest in numerical
analysis. For example, if one is analysing an approximation scheme
based on a finite element method, the spatial regularity determines
the order of convergence. Hence, a suboptimal regularity result leads
to a suboptimal estimate of the order of convergence
(c.f.\ \cite{thomee2006}).

Evolution equations of the form \eqref{eq1:SPDE} are also studied by other
authors. We refer to \cite{daprato1992,krylov1979,rozovski1990,walsh1986} and
the references therein. A related result is \cite{zhang2007}, where
conditions for spatial $C^\infty$-regularity are given.      

The optimal regularity of stochastic convolutions of the form
\begin{align*}
  W_A^\Phi(t) = \int_{0}^{t} E(t - \sigma) \Phi(\sigma) \dWs, 
\end{align*}
is studied in \cite[Prop.~6.18]{daprato1992} and \cite{daprato1982}.
Here $E(t)=\ee^{-tA}$ is an analytic semigroup and $\Phi$ is a
stochastically square integrable ($p=2$) process with values in the
set of Hilbert-Schmidt operators. If, for $r \ge 0$, the process
$\Phi$ is regular enough so that the process $t \mapsto
A^{\frac{r}{2}} \Phi(t)$ is still stochastically square integrable,
then the convolution is a stochastic process, which is square
integrable with values in $\dot{H}^{1+r}$. There exist some
generalizations of this result, for instance, to Banach space valued
integrands \cite{daprato1998}, to the case $p>2$ \cite{neerven2010},
and to L\'{e}vy noise \cite{brzezniak2009}.

Our regularity result for the mild solution of \eqref{eq1:SPDE}
coincides with the optimal regularity property of the stochastic
convolution but with the restriction $r<1$. In this sense we
understand our result to be optimal.

This paper consists of four additional sections. In the next section
we give a more precise formulation of our assumptions. In Section
\ref{sec:space} we are concerned with the spatial regularity of the
mild solution. The proof is divided into several lemmas, which contain
the key ideas of proof.  The lemmas are also useful in the proof of
the temporal H\"older continuity in Section \ref{sec:time}. The proof
of continuity in the border case requires an additional argument in
form of Lebesgue's dominated convergence theorem.  This technique is
also developed in Section \ref{sec:time}. The last section briefly
reviews our result in the special case of additive noise and gives an
example which demonstrates that the spatial regularity results are
indeed optimal.

\sect{Preliminaries}
\label{sec:intro}
In this section we present the general form of the SPDE we are interested in.
After introducing some notation we state our assumptions and cite the result
on existence, uniqueness and regularity of a mild solution from
\cite{jentzen2010b}.

By $H$ we denote a separable Hilbert space $\left(H,( \cdot, \cdot ), \| \cdot
\| \right)$. Further, let $A \colon D(A) \subset H \to H$ be a densely defined,
linear, self-adjoint, positive definite operator, which is not
necessarily bounded but with compact inverse. Hence, there exists an increasing
sequence of real numbers $(\lambda_n)_{n \ge 1}$ and an orthonormal basis
$(e_n)_{n \ge 1}$ in $H$ such that $Ae_n = \lambda_n e_n$ and 
\begin{align*}
  0 < \lambda_1 \le \lambda_2 \le \ldots \le \lambda_n (\to \infty).
\end{align*}
The domain of $A$ is characterized by 
\begin{align*}
  D(A) = \Big\{ x \in H \, : \, \sum_{n = 1}^{\infty} \lambda_n^2 (x,e_n)^2
  < \infty
  \Big\}. 
\end{align*}
Thus, $-A$ is the generator of an analytic semigroup of contractions,
which is denoted by $E(t) = \ee^{-At}$.

By $W \colon [0,T] \times \Omega \to U$ we denote a $Q$-Wiener process
with values in a separable Hilbert space $(U, (\cdot, \cdot)_U, \|
\cdot \|_U)$.  While our underlying probability space is $(\Omega, \F,
P)$, we assume that the Wiener process is adapted to a normal
filtration $(\F_t)_{t \in [0,T]} \subset \F$. The covariance operator
$Q \colon U \to U$ is linear, bounded, self-adjoint, positive
semidefinite but not necessarily of finite trace.

We study the regularity properties of a stochastic
process $X \colon [0,T] \times \Omega \to H$, $T > 0$, which is
the mild solution to the stochastic partial differential equation
\eqref{eq1:SPDE}. Thus, $X$ satisfies the equation
\begin{align}
  X(t) = E(t) X_0 - \int_0^t E(t-\sigma) F(X(\sigma)) \ds + \int_0^t
  E(t-\sigma) G(X(\sigma)) \dWs 
  \label{eq1:mild}
\end{align}
for all $0 \le t \le T$.

In order to formulate our assumptions and main result we introduce the
notion of fractional powers of the linear operator $A$.  For any $r
\in \R$ the operator $A^{\frac{r}{2}}$ is given by
\begin{align*}
  A^{\frac{r}{2}} x = \sum_{n = 1}^\infty \lambda_n^{\frac{r}{2}} x_n e_n 
\end{align*}
for all 
\begin{multline*}
  x \in D(A^{\frac{r}{2}}) 
  = \Big\{ x = \sum_{n = 1}^\infty x_n e_n \, : \,
  (x_n)_{n \ge 1} \subset \R \\ \text{ with } \|x \|_{r}^2 
  := \| A^{\frac{r}{2}} x
  \|^2=\sum_{n=1}^\infty \lambda_n^{r} x_n^2 < \infty \Big\}.
\end{multline*}
By defining $\dot{H}^r := D( A^{\frac{r}{2}} )$ together with the norm
$\| x \|_r$ for $r \in \R$, $\dot{H}^r$ becomes a
Hilbert space. 

As usual \cite{daprato1992,roeckner2007} we introduce the separable Hilbert
space $U_0 := Q^{\frac{1}{2}}(U)$ with the inner product $(u_0,v_0)_{U_0} :=
(Q^{-\frac{1}{2}} u_0, Q^{-\frac{1}{2}} v_0)_U$ with $Q^{-\frac{1}{2}}$
denoting the pseudoinverse. The diffusion operator $G$ maps $H$ into $L^0_2$,
where $L^0_2$ denotes the space of all Hilbert-Schmidt operators $\Phi \colon
U_0 \to H$ with norm
\begin{align*}
  \| \Phi \|_{L^0_2}^2 := \sum_{m = 1}^\infty \| \Phi \psi_m \|^2.
\end{align*}
Here $(\psi_m)_{m \ge 1}$ is an arbitrary orthonormal basis of $U_0$ (for
details see, for example, Proposition 2.3.4 in \cite{roeckner2007}). 
Further, $L_{2,r}^0$ denotes the
set of all Hilbert-Schmidt operators $\Phi \colon U_0 \to \dot{H}^r$ together
with the norm $\| \Phi \|_{L^0_{2,r}} := \| A^{\frac{r}{2}} \Phi \|_{L^0_2}$.

Let $r \in [0,1)$, $p\in [2,\infty)$ be given.
As in \cite{jentzen2010b,printems2001} we make the following
additional assumptions. 

\begin{assumption}
  \label{as:1}
  There exists a constant $C$ such that
  \begin{align}
    \| G(x) - G(y) \|_{L^0_2} \le C \| x - y \| \quad \forall x,y \in H    
    \label{eq1:G_lip}
  \end{align}
  and we have that $G(\dot{H}^r)
  \subset L_{2,r}^0$ and 
  \begin{align}
    \| G(x) \|_{L_{2,r}^0} \le C \left( 1 + \| x \|_{r} \right) \quad \forall x
    \in \dot{H}^r.
    \label{eq1:G_lin}
  \end{align}
\end{assumption}

\begin{assumption}
  \label{as:3}
  The nonlinearity
  $F$ maps $H$ into $\dot{H}^{-1+r}$. Furthermore, there exists a
  constant $C$ such that
  \begin{align}
    \| F(x) -F(y) \|_{-1+r} \le C \| x
    - y \| \quad \forall x,y \in H.
    \label{eq1:F_lip}
  \end{align}
\end{assumption}

\begin{assumption}
  \label{as:2}
  The initial
  value $X_0\colon \Omega \to \dot{H}^{r+1}$ is an $\F_0$-measurable random
  variable with $\E\left[ \| X_0 \|^p_{r+1}\right] <\infty$.
\end{assumption}

Under the above conditions Theorem $1$ in \cite{jentzen2010b} states
that for every $\gamma \in [r, r + 1)$ and $T>0$ there exists an up to
modification unique mild solution $X \colon [0,T] \times \Omega \to
\dot{H}^\gamma$ to \eqref{eq1:SPDE} of the form \eqref{eq1:mild},
which satisfies $$\sup_{t \in [0,T]} \E\left[ \| X(t) \|^p_{\gamma}
\right] < \infty.$$ Moreover, the solution process is continuous with
respect to $\big(\E \left[ \| \cdot \|_{\gamma}^p \right]
\big)^{\frac{1}{p}}$ and fulfills
\begin{align*}
  \sup_{t_1,t_2 \in [0,T], t_1 \neq t_2} \frac{\big(\E \left[ \| X(t_1) -
  X(t_2) \|_{s}^p\right] \big)^{\frac{1}{p}} }{|t_1
  -t_2|^{\min(\frac{1}{2},\frac{\gamma - s}{2})}} < \infty 
\end{align*}
for every $s \in [0,\gamma]$. 

The aim of this paper is to show that these regularity results also hold with
$\gamma = r +1$.  

\begin{remarks}
  1. Actually, Theorem $1$ in \cite{jentzen2010b} assumes that $F \colon H \to
  H$ is globally Lipschitz, which is slightly stronger
  than Assumption \ref{as:3}. That Assumption \ref{as:3} is sufficient can be
  proved by just following the given proof line by line and making the
  appropriate changes where ever $F$ comes into play. 

  2. The linear growth bound \eqref{eq1:G_lin} follows from \eqref{eq1:G_lip}
  when $r = 0$.

      3. Assumption \ref{as:2} can be relaxed to $X_0 \colon \Omega \to H$
      being an 
      $\F_0$-measurable random variable with $\E\left[ \| X_0 \|^p \right] <
      \infty$. But, 
      as it is known from deterministic PDE theory, this will lead to a 
      singularity at $t=0$. 

      4. The framework is quite general. More explicit examples and a
      detailed 
      discussion of Assumption \ref{as:1} can be found in \cite{jentzen2010b}.
      We also refer to the discussion in
      \cite{printems2001} for further examples, references
      and a related result for temporal regularity. 
\end{remarks}

\sect{Spatial regularity}
\label{sec:space}
In this section we deal with the spatial regularity of the mild solution. Our
result is given by the following theorem. For a
more convenient notation we set $\| \cdot \|_{L^p(\Omega;\mathcal{H})} := 
\left( \E \left[ \| \cdot \|^p_{\mathcal{H}} \right] \right)^{\frac{1}{p}}$ for
any Hilbert space $\mathcal{H}$.
Also, if applied to an operator, the norm $\| \cdot \|$ is understood as the
operator norm for bounded, linear operators from $H$ to $H$. 

\begin{theorem}[Spatial regularity]
  \label{prop:1}
  Let $r\in[0,1)$, $p\in[2,\infty)$. Given the assumptions of Section
  \ref{sec:intro} the unique mild solution $X$ in \eqref{eq1:mild}
  satisfies
  \begin{multline*}
    \sup_{t \in [0,T]} \big( \E \big[ \big\| X(t) \big\|_{r+1}^p \big]
    \big)^{\frac{1}{p}} \le \big( \E \big[ \big\| X_0 \big\|_{r+1}^p \big]
    \big)^{\frac{1}{p}} + C \Big( 1 + \sup_{t \in [0,T]} \big( \E \big[
    \big\| X(t) \big\|_{r}^p \big] \big)^{\frac{1}{p}} \Big),   
  \end{multline*}
  where the constant $C$ depends on $p$, $r$, $A$, $F$, $G$, $T$ and the
  H\"older continuity constant of $X$ with respect to the norm $\| \cdot
  \|_{\LOH}$. 

  In particular, $X$ maps into $\dot{H}^{r+1}$ almost surely.
\end{theorem}

Before we prove the theorem we introduce several useful lemmas.
The first states some well known facts on
analytic semigroups (c.f.\ \cite{pazy1983}). Since parts \emph{(iii)},
\emph{(iv)} are not readily found in the literature, we provide proofs here.

\begin{lemma}
  \label{lem:1}
  For the analytic semigroup $E(t)$ the following properties hold true:

  (i) For any $\mu \ge 0$ there exists a constant $C = C(\mu)$ such that
  \begin{align*}
    \| A^\mu E(t) \| \le C t^{-\mu} \text{ for } t > 0.
  \end{align*}

  (ii) For any $0\le \nu \le 1$ there exists a constant $C = C(\nu)$ such
  that
  \begin{align*}
    \| A^{-\nu}(E(t) - I) \| \le C t^\nu \text{ for } t \ge 0.
  \end{align*}

  (iii) For any $0 \le \rho \le 1$ there exists a constant $C = C(\rho)$ such
  that
  \begin{align*}
    \int_{\tau_1}^{\tau_2} \| A^{\frac{\rho}{2}} E(\tau_2 - \sigma) x \|^2 \ds
    \le C (\tau_2 - \tau_1)^{1-\rho} \left\| x 
    \right\|^2 \text{ for all $x \in H$,
    $0 \le \tau_1 < \tau_2$}.
  \end{align*}

  (iv) For any $0 \le \rho \le 1$ there exists a constant $C = C(\rho)$ such
  that
  \begin{align*}
    \Big\| A^{\rho} \int_{\tau_1}^{\tau_2} E(\tau_2 - \sigma) x \ds
    \Big\| \le C (\tau_2 - \tau_1)^{1-\rho} \| x \| \text{ for all $x
    \in H$, $0 \le \tau_1 < \tau_2$}.
  \end{align*}
\end{lemma}

\begin{proof}
  \emph{(iii)} We use the expansion of $x\in H$ in terms
  of the eigenbasis $(e_n)_{n \ge 1}$ of the operator $A$. By Parseval's
  identity we get 
  \begin{align*}
    \int_{\tau_1}^{\tau_2} \big\| A^{\frac{\rho}{2}} E(\tau_2 - \sigma) x
    \big\|^2 \ds
    &= \int_{\tau_1}^{\tau_2} \Big\| \sum_{n = 1}^\infty
    A^{\frac{\rho}{2}} E(\tau_2 - \sigma) (x,e_n) e_n \Big\|^2 \ds\\
    &= \sum_{n=1}^\infty \int_{\tau_1}^{\tau_2} (x,e_n)^2 \lambda_n^\rho
    \ee^{-2\lambda_n (\tau_2 - \sigma)} \ds\\
    &= \frac{1}{2} \sum_{n = 1}^\infty (x,e_n)^2 \lambda_n^{\rho - 1}
    \left(1 - \ee^{-2 \lambda_n (\tau_2 - \tau_1)}\right).
  \end{align*}
  For all $\kappa \in [0,1]$ the function $x
  \mapsto \frac{1 - \ee^{-2x}}{x^\kappa}$ is bounded for $x \in [0, \infty)$.
  Hence,
  \begin{align*}
    0 \le \frac{1 - \ee^{-2 \lambda_n(\tau_2 - \tau_1)}}{( \lambda_n (\tau_2 -
    \tau_1) )^{1-\rho}} \le C(\rho)
  \end{align*}
  for some constant $C(\rho)$, which depends only on $\rho$. Therefore,
  \begin{align*}
    \int_{\tau_1}^{\tau_2} \big\| A^{\frac{\rho}{2}} E(\tau_2 - \sigma) x
    \big\|^2 \ds &\le \frac{1}{2} C(\rho) (\tau_2 - \tau_1)^{1 - \rho}
    \left\| x  \right\|^2. 
  \end{align*}
  The proof of \emph{(iv)} works in a similar way. We square the
  left-hand side and use Parseval's identity again. This yields
  \begin{align*}
    \Big\| A^\rho \int_{\tau_1}^{\tau_2} E(\tau_2 - \sigma) x \ds \Big\|^2
    &= \Big\|  A^\rho \int_{\tau_1}^{\tau_2} \sum_{n =1}^\infty E(\tau_2 -
    \sigma) (x,e_n) e_n \ds \Big\|^2\\
    &= \sum_{n=1}^\infty \lambda_n^{2\rho} \Big( \int_{\tau_1}^{\tau_2} (x,e_n)
    \ee^{-\lambda_n(\tau_2 - \sigma)} \ds \Big)^2\\
    &= \sum_{n=1}^\infty (x,e_n)^2 \Big( \frac{1 - \ee^{-\lambda_n(\tau_2 -
    \tau_1)}}{\lambda_n^{1-\rho}} \Big)^2.
  \end{align*}
  As above we conclude
  \begin{align*}
    \Big\| A^\rho \int_{\tau_1}^{\tau_2} E(\tau_2 - \sigma) x \ds \Big\|^2
    \le C(\rho)^2 (\tau_2 - \tau_1)^{2(1-\rho)} \| x \|^2. 
  \end{align*}
The proof is complete. 
\end{proof}

The next lemma is a special case of Lemma 7.2 in \cite{daprato1992} and will
be needed to estimate the stochastic integrals. 

\begin{lemma}
  \label{lem:stoch_int}
  For any $p \ge 2$, $0 \le \tau_1 < \tau_2 \le T$, and for any 
  $L_2^0$-valued predictable process $\Phi(t)$, $t\in [\tau_1,\tau_2]$, we have
  \begin{align*}
    \E \Big[ \Big\| \int_{\tau_1}^{\tau_2} \Phi(\sigma) \dWs
    \Big\|^{p} \Big] \le C(p) \E\Big[ \Big( \int_{\tau_1}^{\tau_2}
    \big\| \Phi(\sigma) \big\|^2_{L_2^0} \ds \Big)^{\frac{p}{2}} \Big].
  \end{align*}
  Here the constant can be chosen to be
  \begin{align*}
    C(p) = \left( \frac{p}{2} ( p - 1)
    \right)^{\frac{p}{2}} \left( \frac{p}{p - 1}
    \right)^{p(\frac{p}{2} - 1)}.
  \end{align*}
\end{lemma}

The following two lemmas contain our main idea of proof and yield the key
estimates.
\begin{lemma}
  \label{lem:2}
  Let $s \in [0,r+1]$, $p \ge 2$, and $Y$ be a
  predictable stochastic process on $[0,T]$ which maps into $\dot{H}^r$ with
  $\sup_{\sigma \in [0,T]} \|A^{\frac{r}{2}} Y(\sigma) \|_{\LOH} < \infty$.
  Then there exists a constant $C = C(p,r,s,A,G)$ such that, for all $\tau_1,
  \tau_2 \in [0,T]$ with $\tau_1 < \tau_2$,
  \begin{multline}
    \label{eq2:1}
    \Big( \E \Big[ \Big(
    \int_{\tau_1}^{\tau_2} \big\| A^{\frac{s}{2}} E(\tau_2 - \sigma)
    G(Y(\tau_2)) \big\|_{L^0_2}^{2} \ds \Big)^{\frac{p}{2}}
    \Big] \Big)^{\frac{1}{p}}\\
    \le C  \Big(1 +  \sup_{\sigma \in [0,T]}
    \left\|A^{\frac{r}{2}} Y(\sigma)
    \right\|_{\LOH}\Big) (\tau_2 -
    \tau_1)^{\min(\frac{1}{2},\frac{1+r-s}{2})}.
  \end{multline}
  If, in addition, for some $\delta > \frac{r}{2}$ there exists $C_\delta$ such
  that
  \begin{align*}
    \left\| Y(t_1) - Y(t_2) \right\|_{\LOH} \le C_\delta |t_2 -
    t_1|^\delta \text{ for all } t_1, t_2 \in [0,T],
  \end{align*}
  then we also have, with $C = C(p,s,G,C_\delta)$, that
  \begin{multline}
    \label{eq2:2}
    \Big( \E \Big[ \Big(
    \int_{\tau_1}^{\tau_2} \left\| A^{\frac{s}{2}} E(\tau_2 - \sigma)
    \big(G(Y(\sigma)) - G(Y(\tau_2))\big) \right\|_{L^0_2}^{2} \ds
    \Big)^{\frac{p}{2}} \Big]  \Big)^{\frac{1}{p}}  \\
    \le \frac{C}{\sqrt{1 + 2 \delta - s}} (\tau_2 -
    \tau_1)^{\frac{1+2\delta-s}{2}}.
  \end{multline}
  In particular, with $C = C(T,\delta,p,r,s,A,G,C_\delta)$ it holds that
  \begin{multline}
    \label{eq2:est_stoch_int}
    \Big\| \int_{\tau_1}^{\tau_2} A^{\frac{s}{2}} E(\tau_2 - \sigma)
    G(Y(\sigma))\dWs \Big\|_{\LOH}\\
    \le C \Big(1 +  \sup_{\sigma \in [0,T]}
    \left\|A^{\frac{r}{2}} Y(\sigma)
    \right\|_{\LOH}\Big) (\tau_2 - \tau_1)^{\min(\frac{1}{2},\frac{1 + r -
    s}{2})}. 
  \end{multline}
\end{lemma}

\begin{proof} 
  First note that, for $0\le \tau_1 < \tau_2 \le T$ fixed, the mapping
  $[\tau_1,\tau_2] \ni \sigma \mapsto A^{\frac{s}{2}} E(\tau_2 - \sigma)
  G(Y(\sigma))$ is a predictable $L^0_2$-valued process. Hence, Lemma
  \ref{lem:stoch_int} is applicable and gives
  \begin{align*}
    &\Big\| \int_{\tau_1}^{\tau_2} A^{\frac{s}{2}} E(\tau_2 - \sigma)
    G(Y(\sigma))\dWs \Big\|_{\LOH}\\ 
    &\quad \le C(p) \Big\| \Big(
    \int_{\tau_1}^{\tau_2} \left\| A^{\frac{s}{2}} E(\tau_2 - \sigma)
    G(Y(\sigma)) \right\|_{L^0_2}^{2} \ds \Big)^{\frac{1}{2}}
    \Big\|_{\LOR}\\
    &\quad \le C(p) \Big\| \Big(
    \int_{\tau_1}^{\tau_2} \left\| A^{\frac{s}{2}} E(\tau_2 - \sigma)
    G(Y(\tau_2)) \right\|_{L^0_2}^{2} \ds \Big)^{\frac{1}{2}}
    \Big\|_{\LOR}\\ 
    &\qquad + C(p) \Big\| \Big(
    \int_{\tau_1}^{\tau_2} \left\| A^{\frac{s}{2}} E(\tau_2 - \sigma)
    \big(G(Y(\sigma)) - G(Y(\tau_2))\big) \right\|_{L^0_2}^{2} \ds
    \Big)^{\frac{1}{2}} \Big\|_{\LOR} \\
    &\quad=: S_1 + S_2.
  \end{align*}
  In the second step we just used the triangle inequality. Now we deal with
  both summands separately. In the first term $S_1$ the time in $G(Y(\tau_2))$
  is fixed. We also notice that $\eta := s-r - \max(0,s-r)\le 0$ and, hence,
  $A^{\frac{\eta}{2}}$ is a bounded linear operator on $H$. Furthermore, since
  $s \in [0,r+1]$ we have $\rho:= \max(0,s-r) \in [0,1]$ and Lemma
  \ref{lem:1} \emph{(iii)} is applicable. By writing $s = \eta + \rho + r$, we
  get
  \begin{align*}
    &\int_{\tau_1}^{\tau_2} \left\| A^{\frac{s}{2}} E(\tau_2 - \sigma)
    G(Y(\tau_2)) \right\|_{L^0_2}^{2} \ds \\
    &\quad = \int_{\tau_1}^{\tau_2} \sum_{m = 1}^\infty \left\|
    A^{\frac{s}{2}} E(\tau_2 - \sigma) G(Y(\tau_2)) \varphi_m
    \right\|^2 \ds\\
    &\quad \le \sum_{m = 1}^\infty \int_{\tau_1}^{\tau_2} \big\|
    A^{\frac{\eta}{2}} \big\|^2 \big\|  A^{\frac{\rho}{2}} 
    E(\tau_2 - \sigma) A^{\frac{r}{2}} G(Y(\tau_2)) \varphi_m \big\|^{2} \ds
    \\ 
    &\quad \le C(s,r) \big\| A^{\frac{\eta}{2}} \big\|^2
    \big\| A^{\frac{r}{2}} G(Y(\tau_2))
    \big\|^2_{L^0_{2}} (\tau_2 - \tau_1)^{\min(1,1+r-s)},   
  \end{align*}
  where $(\varphi_m)_{m \ge 1}$ denotes an orthonormal basis of $U_0$. We
  also used that $1 - \rho = 1 - \max(0,s-r) = \min(1,
  1+r-s)$. 
  Finally, by Assumption \ref{as:1} we conclude
  \begin{align*}
    S_1 &\le C(p,r,s,A,G) \Big(1 +  \sup_{\sigma \in [0,T]}
    \left\|A^{\frac{r}{2}} Y(\sigma)
    \right\|_{\LOH}\Big) (\tau_2 - \tau_1)^{\min(\frac{1}{2},\frac{1+r-s}{2})}.
 \end{align*}
 This proves \eqref{eq2:1}. For $S_2$ we first make use of
 the fact that $\| B  \Phi \|_{L^0_2} \le \| B \| \| \Phi \|_{L^0_2}$ and then
 apply Lemma \ref{lem:1} \emph{(i)} followed by \eqref{eq1:G_lip}
 to get
 \begin{align*}
   S_2 &\le C(p,s,G) \Big\| \Big( \int_{\tau_1}^{\tau_2} (\tau_2 -
   \sigma)^{-s} \| 
   Y(\sigma) - Y(\tau_2) \|^2 \ds \Big)^{\frac{1}{2}} \Big\|_{\LOR} \\
   &= C(p,s,G) \Big( \Big\| \int_{\tau_1}^{\tau_2} (\tau_2 - \sigma)^{-s} \|
   Y(\sigma) - Y(\tau_2) \|^2 \ds 
   \Big\|_{L^{p/2}(\Omega;\mathbb{R})} \Big)^{\frac{1}{2}}\\
   &\le C(p,s,G) \Big( \int_{\tau_1}^{\tau_2} (\tau_2 - \sigma)^{-s} \|
   Y(\sigma) - Y(\tau_2) \|^2_{\LOH} \ds \Big)^{\frac{1}{2}}.
 \end{align*}
 By the H\"older continuity of $Y$ we
 arrive at 
 \begin{align*}
   S_2 &\le C(p,s,G,C_\delta) \Big( \int_{\tau_1}^{\tau_2} (\tau_2 -
   \sigma)^{-s + 2\delta}  \ds \Big)^{\frac{1}{2}}\\
   &\le \frac{C(p,s,G,C_\delta)}{\sqrt{1 + 2 \delta - s}} (\tau_2 -
   \tau_1)^{\frac{1+2\delta-s}{2}}.
 \end{align*}
 This shows \eqref{eq2:2}. Combination of the estimates for $S_1$ and $S_2$
 yields \eqref{eq2:est_stoch_int} by using $(\tau_2 -
  \tau_1)^\delta \le T^{\delta}$. 
\end{proof}

\begin{lemma}
  \label{lem:3}
  Let $s \in [0,r+1]$, $p \ge 2$, and $Y$ be a
  stochastic 
  process on $[0,T]$ which maps into $H$ with $\sup_{\sigma \in [0,T]}
  \left\| Y(\sigma) \right\|_{\LOH} < \infty$. Then there exists a
  constant $C = C(r,s,F)$ such that, for all $\tau_1,\tau_2 \in [0,T]$ with
  $\tau_1 < \tau_2$,
  \begin{multline}
    \Big\| A^{\frac{s}{2}} \int_{\tau_1}^{\tau_2} E(\tau_2 - \sigma)
    F(Y(\tau_2)) \ds \Big\|_{\LOH}\\
    \le C \Big( 1 + \sup_{\sigma \in [0,T]} \left\| Y(\sigma) \right\|_{\LOH}
    \Big) (\tau_2 - \tau_1)^{\frac{1 + r - s}{2}}. 
    \label{eq2:3}
  \end{multline}
  If, in addition, for some $\delta > 0$ there exists $C_\delta$ such that
  \begin{align*}
    \left\| Y(t_1) - Y(t_2) \right\|_{\LOH} \le C_\delta |t_2 -
    t_1|^\delta \text{ for all } t_1,t_2 \in [0, T],
  \end{align*}
  then we also have, with $C = C(r,s,F,C_\delta)$, that
  \begin{multline}
    \Big\| A^{\frac{s}{2}} \int_{\tau_1}^{\tau_2} E(\tau_2 - \sigma) \big(
    F(Y(\tau_2)) - F(Y(\sigma)) \big) \ds \Big\|_{\LOH} \\
    \le \frac{
    C}{1 + r - s + 2\delta} (\tau_2 - \tau_1)^{\frac{1 + r - s +
    2\delta}{2}}.
    \label{eq2:4}
  \end{multline}
  In particular, with $C = C(T,\delta,r,s,F,C_\delta)$ it holds that
  \begin{multline}
    \label{eq2:est_det_int}
    \Big\| A^{\frac{s}{2}} \int_{\tau_1}^{\tau_2} E(\tau_2 - \sigma)
    F(Y(\sigma))\ds \Big\|_{\LOH}\\
    \le C \Big(1 +  \sup_{\sigma \in [0,T]}
    \left\| Y(\sigma)
    \right\|_{\LOH}\Big) (\tau_2 - \tau_1)^{\frac{1 + r - s}{2}}.
  \end{multline}
\end{lemma}

\begin{proof}
  As in the previous lemma the main idea is to use the H\"older
  continuity of $Y$ to estimate the left-hand side in
  \eqref{eq2:est_det_int}. We have
  \begin{align*}
    &\Big\| A^{\frac{s}{2}} \int_{\tau_1}^{\tau_2} E(\tau_2 - \sigma)
    F(Y(\sigma))\ds \Big\|_{\LOH}\\
    &\quad\le \Big\| A^{\frac{s}{2}} \int_{\tau_1}^{\tau_2} E(\tau_2 - \sigma)
    F(Y(\tau_2)) \ds \Big\|_{\LOH}\\
    &\qquad + \Big\| A^{\frac{s}{2}} \int_{\tau_1}^{\tau_2} E(\tau_2 - \sigma)
    \big( F(Y(\tau_2)) - F(Y(\sigma)) \big) \ds \Big\|_{\LOH}.    
  \end{align*}
  Therefore, if we show \eqref{eq2:3} and \eqref{eq2:4} then
  \eqref{eq2:est_det_int} follows immediately by using $(\tau_2 -
  \tau_1)^\delta \le T^{\delta}$.

  For \eqref{eq2:3} first note that the random variable $A^{\frac{-1+r}{2}}
  F(X(\tau_2))$ takes values in $H$ almost surely. Hence, we can apply Lemma
  \ref{lem:1} \emph{(iv)}. Together with Assumption \ref{as:3} this yields
  \begin{align*}
    &\Big\| A^{\frac{s}{2}} \int_{\tau_1}^{\tau_2} E(\tau_2 - \sigma)
    F(Y(\tau_2)) \ds \Big\|_{\LOH}\\ 
    &\quad \le \Big\| A^{\frac{s+1-r}{2}} \int_{\tau_1}^{\tau_2} E(\tau_2 -
    \sigma) A^{\frac{-1+r}{2}} F(Y(\tau_2)) \ds \Big\|_{\LOH} \\
    &\quad \le C(r,s) (\tau_2 - \tau_1)^{\frac{1+r-s}{2}} \big\|
    A^{\frac{-1+r}{2}} F(Y(\tau_2)) \big\|_{\LOH} \\
    &\quad \le C(r,s,F) \Big( 1 + \sup_{\sigma \in [0,T]} \left\| Y(\sigma)
    \right\|_{\LOH} \Big) (\tau_2 - \tau_1)^{\frac{1+r-s}{2}}.
  \end{align*}
  Finally, again by Lemma \ref{lem:1} and Assumption \ref{as:3}, we show
  \eqref{eq2:4}:
  \begin{align*}
    & \Big\| A^{\frac{s}{2}} \int_{\tau_1}^{\tau_2} E(\tau_2 - \sigma)
    \big( F(Y(\tau_2)) - F(Y(\sigma)) \big) \ds \Big\|_{\LOH} \\
    &\quad\le\int_{\tau_1}^{\tau_2} \big\|  A^{\frac{s + 1 -r}{2}} E(\tau_2 -
    \sigma) A^{\frac{-1+r}{2}} \left( F(Y(\tau_2)) - F(Y(\sigma)) \right)
    \big\|_{\LOH} \ds \\ 
    &\quad\le C(r,s,F) \int_{\tau_1}^{\tau_2}  (\tau_2 -
    \sigma)^{\frac{r-s-1}{2}} 
    \left\| Y(\tau_2) - Y(\sigma) \right\|_{\LOH} \ds\\
    &\quad\le C(r,s,F,C_\delta) \int_{\tau_1}^{\tau_2} (\tau_2 -
    \sigma)^{\frac{r-s-1 + 2\delta}{2}} \ds =
    \frac{2C(r,s,F,C_\delta)}{1 + r - s + 2\delta} (\tau_2 -
    \tau_1)^{\frac{1 + r-s + 2\delta}{2}}.
  \end{align*}
This completes the proof. 
\end{proof}

Now we are well prepared for the proof of Theorem \ref{prop:1}.
\begin{proof}[Proof of Theorem \ref{prop:1}]
  By taking norms in \eqref{eq1:mild} we get, for $t\in [0,T]$,
  \begin{align*}
    \left(\E \left[ \| X(t) \|^p_{r+1}\right] \right)^{\frac{1}{p}} &= \|
    A^{\frac{r+1}{2}} X(t) \|_{\LOH} \\
    &\le \| A^{\frac{r+1}{2}} E(t) X_0  \|_{\LOH} \\ 
    &\quad + 
    \Big\| A^{\frac{r+1}{2}} \int_{0}^{t} E(t-\sigma) F(X(\sigma))\ds
    \Big\|_{\LOH}\\ 
    &\quad + \Big\| A^{\frac{r+1}{2}} \int_{0}^{t} E(t-\sigma)
    G(X(\sigma)) \dWs \Big\|_{\LOH}\\
    &=: I + II + III.
  \end{align*}
  The first term is well-known from deterministic theory and can be estimated
  by 
  \begin{align*}
    \| A^{\frac{r+1}{2}} E(t) X_0  \|_{\LOH} \le \| A^{\frac{r+1}{2}} X_0
    \|_{\LOH} < \infty,     
  \end{align*}
  since $X_0\colon \Omega \to \dot{H}^{r+1}$ by Assumption \ref{as:2}.

  We recall that, by Theorem $1$ in \cite{jentzen2010b}, the mild
  solution $X$ is an $\dot{H}^r$-valued predictable stochastic process
  which is $\delta$-H\"older continuous for any $0 < \delta <
  \frac{1}{2}$ with respect to the norm $\| \cdot \|_{\LOH}$. We
  choose $\delta := \frac{r+1}{4}$ so that $0\le\frac{r}{2} < \delta <
  \frac{1}{2}$. Hence, we can apply Lemmas \ref{lem:2} and \ref{lem:3}
  with $Y=X$.
 
  For the second term we apply \eqref{eq2:est_det_int} with
  $\tau_1 = 0$, $\tau_2 = t$, $s = r+1$ and $Y = X$. This yields
  \begin{align*}
    II &\le C \Big( 1 +
    \sup_{\sigma \in [0,T]} \left\| X(\sigma) \right\|_{\LOH} \Big)
    < \infty.
  \end{align*}
  For the last term we apply \eqref{eq2:est_stoch_int} with
  the same parameters as above: 
  \begin{align*}
    III &\le C \Big(1 +  \sup_{\sigma \in [0,T]}
    \left\|A^{\frac{r}{2}} X(\sigma)
    \right\|_{\LOH}\Big) < \infty. 
  \end{align*}
  Note that $\sup_{\sigma \in [0,T]}  \| X(\sigma)
  \|_{\LOH}\le \sup_{\sigma \in [0,T]}  \|A^{\frac{r}{2}} X(\sigma)
  \|_{\LOH}$ is finite because of Theorem $1$ in \cite{jentzen2010b}.
\end{proof}

\sect{Regularity in time}
\label{sec:time}

This section is devoted to the temporal regularity of the mild solution. Our
main result is summarized in the following theorem. For the border case $s=r+1$
we refer to Theorem \ref{th:3} below.

\begin{theorem}[Temporal regularity]
  \label{prop:2}
  Let $r \in [0,1)$, $p\in[2,\infty)$. Under the assumptions of
  Section \ref{sec:intro} the unique mild solution $X$ to \eqref{eq1:SPDE} is
  H\"older continuous 
  with respect to $\big( \E \left[ \| \cdot \|^p_{s} \right]
  \big)^{\frac{1}{p}}$ and satisfies
  \begin{align}
    \label{eq3:1}
    \sup_{t_1,t_2 \in [0,T], t_1 \neq t_2} \frac{\big( \E \left[ \| X(t_1) -
    X(t_2) \|_{s}^p\right] \big)^{\frac{1}{p}} }{|t_1
    -t_2|^{\min(\frac{1}{2},\frac{1 +r- s}{2})}}
    < \infty 
  \end{align}
for every $s \in [0,r+1)$.  
\end{theorem}

\begin{proof}%[Proof of Theorem \ref{prop:2}]
  Let $0 \le t_1 < t_2 \le T$ be arbitrary. 
  By using the mild formulation \eqref{eq1:mild} we get
  \begin{align}
    &\big( \E \left[ \left\| X(t_1) - X(t_2) \right\|_{s}^p
    \right] \big)^{\frac{1}{p}} = \left\| A^{\frac{s}{2}} (X(t_1) - X(t_2))
    \right\|_{\LOH} \notag\\
    &\quad \le \left\| A^{\frac{s}{2}} ( E(t_1) - E(t_2) ) X_0 \right\|_{\LOH}
    \notag\\
    &\qquad + \Big\| A^{\frac{s}{2}} \int_{t_1}^{t_2} E(t_2 - \sigma)
    F(X(\sigma)) \ds \Big\|_{\LOH} \notag \\ 
    &\qquad + \Big\|A^{\frac{s}{2}}  \int_{0}^{t_1} \left( E(t_2 - \sigma) -
    E(t_1 - \sigma) \right) F(X(\sigma)) \ds\Big\|_{\LOH} \notag\\
    &\qquad + \Big\| A^{\frac{s}{2}} \int_{t_1}^{t_2} E(t_2 - \sigma)
    G(X(\sigma)) 
    \dWs \Big\|_{\LOH} \notag \\
    &\qquad + \Big\| A^{\frac{s}{2}} \int_{0}^{t_1} \left( E(t_2 - \sigma) -
    E(t_1 - \sigma) \right) G(X(\sigma)) \dWs \Big\|_{\LOH}\notag \\
    &\quad=: T_1 + T_2 + T_3 + T_4 + T_5. 
    \label{eq4:0}
  \end{align}
  We estimate the five terms separately. The term $T_1$ is estimated by 
  \begin{align*}
    T_1 &= \left\| A^{\frac{s-r-1}{2}} ( I - E(t_2 - t_1))
    A^{\frac{r+1}{2}} E(t_1) X_0 \right\|_{\LOH} \\
    &\le C \big\| A^{\frac{r+1}{2}} X_0
    \big\|_{\LOH} (t_2 - t_1)^{\frac{1+r-s}{2}},
  \end{align*}
  where we used Lemma \ref{lem:1} \emph{(ii)} and Assumption \ref{as:2}.

  As in the proof of Theorem \ref{prop:1} we choose the H\"older
  exponent $\delta := \frac{r+1}{4}$ so that $\frac{r}{2} < \delta <
  \frac{1}{2}$ and we can apply Lemmas \ref{lem:2} and \ref{lem:3}
  with $Y=X$.
  
  The term $T_2$ coincides with
  \eqref{eq2:est_det_int} and we have  
  \begin{align*}
    T_2 \le C \Big( 1 + \sup_{\sigma \in [0,T]} \| X(\sigma) \|_{\LOH} \Big)
    (t_2 - t_1)^{\frac{1 + r - s}{2}}.
  \end{align*}
  For the third term we also apply Lemma \ref{lem:1} \emph{(ii)} before
  we use \eqref{eq2:est_det_int}: 
  \begin{align*}
    T_3 &= \Big\| A^{\frac{s - r - 1}{2}} \left( E(t_2 - t_1) - I \right)
    A^{\frac{r+1}{2}} \int_{0}^{t_1} E(t_1 - \sigma) F(X(\sigma)) \ds
    \Big\|_{\LOH} \\
    &\le C (t_2 - t_1)^{\frac{1+r-s}{2}} \Big\| A^{\frac{r+1}{2}} 
    \int_{0}^{t_1}
    E(t_1 - \sigma) F(X(\sigma)) \ds
    \Big\|_{\LOH}\\
    &\le C \Big( 1 + \sup_{\sigma \in [0,T]} \left\| X(\sigma)
    \right\|_{\LOH} \Big) (t_2 - t_1)^{\frac{1+r-s}{2}}.
  \end{align*}
  The fourth term is estimated analogously by using
  \eqref{eq2:est_stoch_int} instead of \eqref{eq2:est_det_int}. We get
  \begin{align*}
    T_4 \le C 
    \Big( 1 +
    \sup_{\sigma \in [0,T]} \left\| A^{\frac{r}{2}} X(\sigma)
    \right\|_{\LOH} \Big) 
    (t_2 - t_1)^{\min(\frac{1}{2},\frac{1+r-s}{2})}.
  \end{align*}
  Finally, for the last term we use Lemma \ref{lem:stoch_int}
  first. Since, for $0 \le t_1 < t_2 \le T$ fixed, the function
  $[0,t_1] \ni \sigma \mapsto A^{\frac{s}{2}} \left( E(t_2 - \sigma) -
    E(t_1 - \sigma)\right) G(X(\sigma))$ is a predictable stochastic
  process Lemma \ref{lem:stoch_int} can be applied. Then, by using
  Lemma \ref{lem:1} \emph{(ii)} with $\nu = \frac{1+r-s}{2}$ and Lemma
  \ref{lem:2} with $s = r+1$ we get
  \begin{align*}
    T_5 &\le C \Big\| \Big( \int_{0}^{t_1} \big\| A^{\frac{s - r - 1}{2}}
    (E(t_2 - t_1) - I) A^{\frac{r+1}{2}} E(t_1 - \sigma) G(X(\sigma))
    \big\|^2_{L_2^0} \ds \Big)^{\frac{1}{2}} \Big\|_{\LOR} \\
    &\le C (t_2 - t_1)^{\frac{1 + r - s}{2}} \Big( 
    \Big\| \Big( \int_{0}^{t_1} \big\| A^{\frac{r+1}{2}} E(t_1 - \sigma)
    G(X(t_1)) \big\|^2_{L_2^0} \ds \Big)^{\frac{1}{2}} \Big\|_{\LOR} \\
    &\quad +\Big\| \Big( \int_{0}^{t_1} \big\| A^{\frac{r+1}{2}} E(t_1 - \sigma)
    \big( G(X(\sigma)) - G(X(t_1)) \big) \big\|^2_{L_2^0} \ds
    \Big)^{\frac{1}{2}} \Big\|_{\LOR} \Big)\\
    &\le C (t_2 - t_1)^{\frac{1 + r - s}{2}} \Big( 1 + \sup_{\sigma \in [0,T]}
    \left\| A^{\frac{r}{2}} X(\sigma) \right\|_{\LOH} \Big). 
  \end{align*}
  Altogether, this proves \eqref{eq3:1} and the H\"older continuity of
  $X$ with respect to the norm $\| A^{\frac{s}{2}} \cdot \|_{\LOH}$ for all $s
  \in [0,r+1)$. 
\end{proof}

The temporal regularity of $X$ with respect to the
norm $\big(\E[ \| \cdot \|^p_{r+1} \big)^{\frac{1}{p}}$ is more involved. For
the case $r=0$ we can prove the following result.
%For the general case $r \in
%(0,1)$ we refer to the remark at the end of this section.

\begin{theorem}
  \label{th:3}
  Let $r=0$ and $p \in [2,\infty)$. Under the assumptions of Section
  \ref{sec:intro} the 
  unique mild solution $X$ to \eqref{eq1:SPDE} is continuous 
  with respect to $\big( \E \left[ \| \cdot \|^p_{1} \right]
  \big)^{\frac{1}{p}}$.
\end{theorem}

Before we begin the proof we analyze the continuity properties of
the semigroup in the deterministic context. 

\begin{lemma}
  \label{lem:4}
  Let $0 \le \tau_1 < \tau_2 \le T$. Then we have

  (i) 
  \begin{align*}
    \lim_{\tau_2 - \tau_1 \to 0} \int_{\tau_1}^{\tau_2} \big\|
    A^{\frac{1}{2}} E(\tau_2 - \sigma) x \big\|^2 \ds = 0 \quad \text{ for all
    } x \in H, 
  \end{align*}

  (ii)
  \begin{align*}
    \lim_{\tau_2 - \tau_1 \to 0} \Big\| A \int_{\tau_1}^{\tau_2}  E(\tau_2 -
    \sigma) x \ds \Big\| = 0 \quad \text{ for all } x
    \in H.
  \end{align*}
%
%  (iii)
%  \begin{align*}
%    \lim_{\tau_2 - \tau_1 \to 0} \int_{0}^{\tau_1} \big\|
%    A^{\frac{1}{2}} \big( E(\tau_2 - \sigma) - E(\tau_1 - \sigma) \big) x 
%    \big\|^2 \ds = 0 \quad \text{ for all } x \in H,
%  \end{align*}
%
%  (iv)
%  \begin{align*}
%    \lim_{\tau_2 - \tau_1 \to 0} \Big\| A \int_{0}^{\tau_1} \big( E(\tau_2
%    - \sigma) - E(\tau_1 - \sigma) \big) x \ds \Big\| = 0 \quad \text{ for all
%    } x \in H. 
%  \end{align*}
\end{lemma}

\begin{proof}
  As in the proof of Lemma \ref{lem:1} we use the orthogonal expansion of $x
  \in H$ with respect to the eigenbasis $(e_n)_{n \ge 1}$ of the operator $A$.
  Thus, for \emph{(i)} we get, as in the proof of Lemma \ref{lem:1}
  \emph{(iii)},
  \begin{align*}
    \int_{\tau_1}^{\tau_2} \left\| A^{\frac{1}{2}} E(\tau_2 - \sigma) x
    \right\|^2 \ds = \frac{1}{2} \sum_{n = 1}^\infty (x,e_n)^2 \left(1 - \ee^{2
    \lambda_n(\tau_2 - \tau_1)} \right).    
  \end{align*}
  We apply Lebesgue's dominated convergence theorem. Note that
  the sum is dominated by
  $\frac{1}{2} \| x \|^2$ for all $\tau_2 - \tau_1 \ge 0$.
  Moreover, for every $n \ge 1$ we have
  \begin{align*}
    \lim_{\tau_2 - \tau_1 \to 0} \big( 1 - \ee^{2\lambda_n(\tau_2 - \tau_1)}
    \big) (x, e_n)^2 = 0.
  \end{align*}
  Hence, Lebesgue's theorem gives us \emph{(i)}. The same argument also yields
  the second case, since
  \begin{align*}
    \Big\| A \int_{\tau_1}^{\tau_2} E(\tau_2 - \sigma) x \ds \Big\|^2 =
    \sum_{n = 1}^\infty (x,e_n)^2\big( 1 - \ee^{\lambda_n(\tau_2 - \tau_1)} \big)^2.  
 \end{align*}
%  and
%  \begin{align*}
%    &\int_{0}^{\tau_1} \big\|
%    A^{\frac{1}{2}} \left( E(\tau_2 - \tau_1) - I \right) E(\tau_1 - \sigma) x 
%    \big\|^2 \ds\\ 
%    &\quad= \sum_{n=1}^\infty \big( \ee^{-\lambda_n(\tau_2 - \tau_1)} -
%    1\big)^2 \int_{0}^{\tau_1} \lambda_n \ee^{-2\lambda_n(\tau_1 - \sigma)} \ds
%    \, (x,e_n)^2\\
%    &\quad= \frac{1}{2} \sum_{n=1}^\infty \big( \ee^{-\lambda_n(\tau_2 -
%    \tau_1)} - 1 \big)^2 \left( 1 - \ee^{-2\lambda_n\tau_1} \right) (x,e_n)^2,
%  \end{align*}
%  as well as
%  \begin{align*}
%    &\Big\| A \int_{0}^{\tau_1} \left( E(\tau_2
%    - \tau_1) - I \right) E(\tau_1 - \sigma) x \ds \Big\|^2\\
%    &\quad=
%    \sum_{n=1}^\infty \Big( \ee^{-\lambda_n(\tau_2 - \tau_1)} - 1 \Big)^2
%    \Big( \lambda_n \int_{0}^{\tau_1} \ee^{-\lambda(\tau_1 - 
%    \sigma)} \ds \Big)^2 (x,e_n)^2 \\
%    &\quad= \sum_{n=1}^\infty \big(\ee^{-\lambda_n(\tau_2 - \tau_1)} - 1
%    \big)^2 \left( 1 - \ee^{-\lambda_n \tau_1} \right) (x,e_n)^2.    
%  \end{align*}
The proof is complete. 
\end{proof}

\begin{proof}[Proof of Theorem \ref{th:3}]
  { We must show that
    $\lim_{t_2-t_1\to0^+}\|X(t_2)-X(t_1)\|_{L^p(\Omega;\dot{H}^1)}=0$
    with either $t_1$ or $t_2$ fixed.  As already demonstrated in the
    proof of Lemma \ref{lem:4} we use Lebesgue's dominated convergence
    theorem. Let $0\le t_1<t_2\le T$.  We consider again the terms
    $T_i$, $i=1,\ldots,5$, in \eqref{eq4:0} but now with $s=1$.

  For $T_1$ continuity follows immediately: For almost every $\omega \in
  \Omega$ we get that $X_0(\omega) \in \dot{H}^{1}$. Thus, for every fixed
  $\omega \in \Omega$ with this property we have
  \begin{align*}
    \lim_{t_2 - t_1 \to 0} \| (E(t_2) - E(t_1))
    A^{\frac{1}{2}} X_0(\omega) \| = 0
  \end{align*}
  by the strong continuity of the semigroup. We also have that $$\| (E(t_2) -
  E(t_1)) A^{\frac{1}{2}} X_0(\omega) \| \le \|A^{\frac{1}{2}}
  X_0(\omega) \|,$$ where the latter is an element of $L^p(\Omega;\R)$ as a
  function of $\omega \in \Omega$ by
  Assumption \ref{as:2}. Hence, Lebesgue's theorem is
  applicable and yields $\lim_{t_2 - t_1 \to 0} T_1 = 0$.

  In order to treat the right and left limits simultaneously in the
  remaining terms, we compute the limits as $t_1\to t_3$ and $t_2\to
  t_3$ for fixed but arbitrary $t_3\in[t_1,t_2]$.

  %Since the same argument applies to the remaining terms, it is sufficient to
  %provide the pointwise limit for fixed $\omega \in \Omega$
  %and a dominating function.
  
 In the case of $T_2$ we get
  \begin{align}
    \label{eq4:T2}
    \begin{split}
      T_2 &\le \Big\| A \int_{t_1}^{t_2} E(t_2 - \sigma)
      A^{-\frac{1}{2}} \big( F(X(\sigma)) - F(X(t_2)) \big) \ds
      \Big\|_{\LOH} \\
      &\quad + \Big\| A \int_{t_1}^{t_2} E(t_2 - \sigma) A^{-\frac{1}{2}} \big(
      F(X(t_2)) - F(X(t_3)) \big) \ds \Big\|_{\LOH} \\
      &\quad + \Big\| A \int_{t_1}^{t_2} E(t_2 - \sigma) A^{-\frac{1}{2}}
      F(X(t_3)) \ds \Big\|_{\LOH}.
    \end{split}
  \end{align}
  Because of \eqref{eq2:4}, where we can choose $s=r+1=1$ and $\delta =
  \frac{1}{4} > 0$, the limit $t_2 - t_1 \to 0$ of the first summand 
  is $0$. For the second summand in \eqref{eq4:T2} we apply Lemma
  \ref{lem:1} \emph{(iv)} with $\rho  =1$, and Assumption \ref{as:3} with
  $r=0$. Then we derive
  \begin{align*}
    & \Big\| A \int_{t_1}^{t_2} E(t_2 - \sigma) A^{-\frac{1}{2}} \big(
    F(X(t_2)) - F(X(t_3)) \big) \ds \Big\|_{\LOH} \\
    & \quad \le C \big\|A^{-\frac{1}{2}} ( F(X(t_2)) - F(X(t_3)) )
    \big\|_{\LOH}\\
    & \quad \le C \|X(t_2) - X(t_3) \|_{\LOH} 
  \end{align*}
  and the limit $t_2 \to t_3$ of this term vanishes by \eqref{eq3:1} with
  $s=0$. 

  For the last summand in \eqref{eq4:T2} we again apply Lemma \ref{lem:1}
  \emph{(iv)} with $\rho  =1$ and obtain, for almost every $\omega \in \Omega$,
  \begin{align*}
   \begin{split}
    \Big\| A \int_{t_1}^{t_2}
    E(t_2 - \sigma) A^{-\frac{1}{2}} F(X(t_3,\omega)) \ds
    \Big\|  
    &\le C \| A^{-\frac{1}{2}} F(X(t_3,\omega)) \| \\
    &\le C \big(1 + \| X(t_3,\omega) \| \big) ,    
  \end{split}
  \end{align*}
  which belongs to $L^p(\Omega;\R)$ for all $t_3 \in [0,T]$. By Lemma
  \ref{lem:4} \emph{(ii)} it also holds that
  \begin{align*}
    \lim_{t_1 \to t_3 \atop t_2 \to t_3} \Big\| A \int_{t_1}^{t_2}
    E(t_2 - \sigma) A^{-\frac{1}{2}} F(X(t_3,\omega)) \ds
    \Big\| = 0    
  \end{align*}
  for almost all $\omega \in \Omega$. Then Lebesgue's dominated convergence
  theorem yields that this term vanishes, which completes the proof for $T_2$.
  
  Next, we take care of $T_3$, which is estimated by
  \begin{align}
    \begin{split}
    T_3 &\le \Big\| A^{\frac{1}{2}} \int_{0}^{t_1} \big( E(t_2 - \sigma) -
    E(t_1 - \sigma)\big) \big( F(X(\sigma)) - F(X(t_1)) \big) \ds
    \Big\|_{\LOH} \\
    &\quad + \Big\| A^{\frac{1}{2}} \int_{0}^{t_1} \big( E(t_2 - \sigma) -
    E(t_1 - \sigma)\big) \big(F(X(t_1)) - F(X(t_3)) \big) \ds \Big\|_{\LOH}\\
    &\quad+ \Big\| A^{\frac{1}{2}} \int_{0}^{t_1} \big( E(t_2 - \sigma) - E(t_1
    - \sigma) \big) F(X(t_3)) \ds \Big\|_{\LOH}.
  \end{split}
  \label{eq4:T3}
  \end{align}
  For the first summand in \eqref{eq4:T3} we get by Lemma \ref{lem:1}
  \emph{(ii)} 
  \begin{align}
    \label{eq4:T31}
    \begin{split}
    &\Big\| A^{\frac{1}{2}} \int_{0}^{t_1} \big( E(t_2 - \sigma) -
    E(t_1 - \sigma)\big) \big( F(X(\sigma)) - F(X(t_1))
    \big) \ds 
    \Big\|_{\LOH}\\
    &\; \le \int_{0}^{t_1} \big\| A^{-\frac{\eta}{2}} \big( E(t_2 - t_1) -
    I \big) A^{\frac{1 + \eta}{2}}
    E(t_1 - \sigma) \big( F(X(\sigma)) - F(X(t_1)) \big)
    \big\|_{\LOH} \ds\\
    &\; \le C (t_2 - t_1)^{\frac{\eta}{2}} \int_{0}^{t_1} (t_1
    -\sigma)^{-\frac{2 + \eta}{2}} \big\| A^{-\frac{1}{2}} 
    \big( F(X(\sigma)) - F(X(t_1)) \big) \big\|_{\LOH} \ds,
  \end{split}
  \end{align}
  where $\eta \in (0,2]$.
  We continue the estimate by applying Assumption \ref{as:3} and the
  H\"older continuity of $X$ with exponent $\frac{1}{2}$ with respect to the
  norm $\| \cdot \|_{\LOH}$ as it was shown in \eqref{eq3:1} with $s=0$. This
  gives 
  \begin{align*}
    &\Big\| A \int_{0}^{t_1} \big( E(t_2 - \sigma) - 
    E(t_1 - \sigma)\big) A^{-\frac{1}{2}}\big( F(X(\sigma)) - F(X(t_1)) \big)
    \ds \Big\|_{\LOH}\\
    &\quad \le C (t_2 - t_1)^{\frac{\eta}{2}} \int_{0}^{t_1} (t_1
    -\sigma)^{-\frac{2 + \eta -1}{2}} \ds = C \frac{2}{1 - \eta} t_1^{\frac{1 -
    \eta}{2}} (t_2 - t_1)^{\frac{\eta}{2}}.
  \end{align*}
  Therefore, in the limit $t_2 - t_1 \to 0$ this term is zero as
  long as $\eta \in (0,1)$.

  For the second summand in \eqref{eq4:T3} we apply Lemma \ref{lem:1}
  \emph{(iv)} with $\rho = 1$ and get
  \begin{align*}
    &\Big\| A^{\frac{1}{2}} \int_{0}^{t_1} \big( E(t_2 - \sigma) -
    E(t_1 - \sigma)\big) F(X(t_1)) \ds \Big\|_{\LOH}\\
    &\quad  = \Big\| A \int_{0}^{t_1} E(t_1 - \sigma) \big( E(t_2 - t_1) -
    I \big) A^{-\frac{1}{2}} F(X(t_1)) \ds \Big\|_{\LOH}\\   
    &\quad \le C \Big\| \big( E(t_2 - t_1) -
    I \big) A^{-\frac{1}{2}} F(X(t_1)) \Big\|_{\LOH}\\
    &\quad \le C\Big\| \big( E(t_2 - t_1) - I \big)
    A^{-\frac{1}{2}}\big(F(X(t_1)) - F(X(t_3)\big) \Big\|_{\LOH} \\ 
    & \qquad+ C \Big\|\big( E(t_2 - t_1) - I \big) A^{-\frac{1}{2}} F(X(t_3))
    \Big\|_{\LOH}. 
  \end{align*}
  By Assumption \ref{as:3} and \eqref{eq3:1} it holds true that
  \begin{align*}
    &\Big\| \big( E(t_2 - t_1) - I \big)
    A^{-\frac{1}{2}}\big(F(X(t_1)) - F(X(t_3)\big) \Big\|_{\LOH}\\
    &\quad \le C \big\| X(t_1) - X(t_3) \big\|_{\LOH} \le C | t_1 -
    t_3|^{\frac{1}{2}}.
  \end{align*}
  Hence, this term vanishes in the limit $t_1 \to t_3$. 
  Therefore, the proof for $T_3$ is complete, if we can show that
  \begin{align*}
    &\lim_{t_1 \to t_3 \atop t_2 \to t_3} \big\| \big( E(t_2 - t_1) - I \big)
    A^{-\frac{1}{2}} F(X(t_3)) \big\|_{\LOH}\\
    &\quad = \lim_{t_1 \to t_3 \atop t_2 \to t_3} \Big\| A \int_{t_1}^{t_2}
    E(t_2 -\sigma) A^{-\frac{1}{2}} F(X(t_3)) \ds \Big\|_{\LOH} = 0.  
  \end{align*}
  This is true by an application of Lebesgue's dominated convergence theorem.
  In order to apply this theorem, we obtain a
  dominating function for almost every $\omega \in \Omega$ by
  \begin{align*}
    \big\| \big( E(t_2 - t_1) - I \big)
    A^{-\frac{1}{2}} F(X(t_3,\omega)) \big\| \le C \big( 1 + \| X(t_3,\omega)
    \| \big).    
  \end{align*}
  Further, Lemma \ref{lem:4} \emph{(ii)} yields
  \begin{align*}
    \lim_{t_1 \to t_3 \atop t_2 \to t_3} \Big\| A \int_{t_1}^{t_2}
    E(t_2 -\sigma) A^{-\frac{1}{2}} F(X(t_3,\omega)) \ds \Big\| = 0
  \end{align*}
  for almost every $\omega \in \Omega$. Altogether, this shows
  $\lim_{t_1 \to t_3 \atop t_2 \to t_3} T_3 = 0$.  
  }

  {
  For $T_4$, one has to use
  Lemma \ref{lem:stoch_int}, which yields
  \begin{align}
    \begin{split}
    T_4
    &\le C \Big\|
    \Big( \int_{t_1}^{t_2} \big\| A^{\frac{1}{2}} E(t_2-\sigma) G(X(\sigma))
    \big\|_{L_2^0}^2 \ds \Big)^{\frac{1}{2}} \Big\|_{\LOR} \\
    &\le C \Big\| \Big( \int_{t_1}^{t_2}
    \big\| A^{\frac{1}{2}} E(t_2-\sigma) \big( G(X(\sigma)) - G(X(t_2)) \big)
    \big\|_{L_2^0}^2 \ds \Big)^{\frac{1}{2}} \Big\|_{\LOR}\\
    &\quad + C \Big\| \Big( \int_{t_1}^{t_2}
    \big\| A^{\frac{1}{2}} E(t_2-\sigma) \big( G(X(t_2)) - G(X(t_3)) \big)
    \big\|_{L_2^0}^2 \ds \Big)^{\frac{1}{2}} \Big\|_{\LOR}\\
    &\quad + C \Big\| \Big( \int_{t_1}^{t_2}
    \big\| A^{\frac{1}{2}} E(t_2-\sigma) G(X(t_3))
    \big\|_{L_2^0}^2 \ds \Big)^{\frac{1}{2}} \Big\|_{\LOR}.
  \end{split}
  \label{eq4:T4}
  \end{align}
  The limit $t_2 -t_1 \to 0$ of the first summand is $0$ because of
  \eqref{eq2:2}, where we again choose $s = 1$ and $\delta = \frac{1}{4} >
  \frac{r}{2} = 0$. As before, we discuss the simultaneous limits $t_1 \to 
  t_3$ and $t_2 \to t_3$ for the remaining summands in \eqref{eq4:T4}.

  By Lemma \ref{lem:1} \emph{(iii)} it holds for the second summand in
  \eqref{eq4:T4} that
  \begin{align}
    \label{eq4:T42}
    \begin{split}
      & \Big\| \Big( \int_{t_1}^{t_2}
      \big\| A^{\frac{1}{2}} E(t_2-\sigma) \big( G(X(t_2)) -
      G(X(t_3)) \big) \big\|_{L_2^0}^2 \ds \Big)^{\frac{1}{2}} \Big\|_{\LOR} \\
      &\quad = \Big\| \Big( \sum_{m=1}^\infty \int_{t_1}^{t_2}
      \big\| A^{\frac{1}{2}} E(t_2-\sigma) \big( G(X(t_2)) -
      G(X(t_3)) \big) \varphi_m \big\|^2 \ds \Big)^{\frac{1}{2}}
      \Big\|_{\LOR} \\
      &\quad \le C \big\| G(X(t_2)) -
      G(X(t_3))  \big\|_{L^p(\Omega;L_{2}^0)} \\
      &\quad \le C \big\| X(t_2) - X(t_3) \big\|_{\LOH}.
    \end{split}
  \end{align}
  Consequently, this term also vanishes as $t_2 \to t_3$ by \eqref{eq3:1}. 
  
  Next we come to the third summand in \eqref{eq4:T4}. By Lemma \ref{lem:1}
  \emph{(iii)} with $\rho = 1$ and 
  Assumption \ref{as:1} 
  we obtain for almost every $\omega \in \Omega$
  \begin{align*}
    &\Big( \int_{t_1}^{t_2}
    \big\| A^{\frac{1}{2}} E(t_2-\sigma) G(X(t_3,\omega))
    \big\|_{L_2^0}^2 \ds \Big)^{\frac{1}{2}} \\
    &\quad = \Big( \sum_{m=1}^\infty \int_{t_1}^{t_2} \big\|
    A^{\frac{1}{2}} E(t_2-\sigma) G(X(t_3,\omega)) \varphi_m
    \big\|^2 \ds \Big)^{\frac{1}{2}} \\    
    &\quad \le C \Big( \sum_{m=1}^\infty  \big\|
    G(X(t_3,\omega)) \varphi_m
    \big\|^2 \Big)^{\frac{1}{2}} \le C \big( 1 + \big\| X(t_3,\omega) \big\|
    \big),
  \end{align*}
  where $(\varphi_m)_{m \ge 1}$ is an arbitrary orthonormal basis of $U_0$ and
  the last term belongs to $L^p(\Omega;\R)$. In order to apply Lebesgue's
  Theorem it remains to discuss the pointwise limit. For this Lemma \ref{lem:4}
  \emph{(i)} yields
  \begin{align*}
    & \lim_{t_1 \to t_3 \atop t_2 \to t_3} \int_{t_1}^{t_2}
    \big\| A^{\frac{1}{2}} E(t_2-\sigma) G(X(t_3,\omega))
    \big\|_{L_2^0}^2 \ds \\
    &\quad = \sum_{m=1}^\infty \lim_{t_1 \to t_3 \atop t_2 \to t_3}
    \int_{t_1}^{t_2} \big\| A^{\frac{1}{2}} E(t_2-\sigma) G(X(t_3,\omega))
    \varphi_m \big\|^2 \ds = 0.
  \end{align*}
  In fact, the interchanging of summation and taking the limit is justified by
  a further application of Lebesgue's Theorem. Altogether, this proves the
  desired result for $T_4$.
  }

  {
  The estimate of $T_5$ works similarly as for $T_3$. We
  apply Lemma \ref{lem:stoch_int} and get
  \begin{align}
    \label{eq4:T5}
    \begin{split}
      T_5 &\le C \Big\| \Big( \int_{0}^{t_1} \big\| A^{\frac{1}{2}} \big(
      E(t_2 - \sigma) - E(t_1 - \sigma) \big) 
      \big( G(X(\sigma)) - G(X(t_1)) \big) \big\|^2_{L_2^0}
      \ds\Big)^{\frac{1}{2}} \Big\|_{\LOR} \\
      &\quad + C \Big\| \Big( \int_{0}^{t_1} \big\| A^{\frac{1}{2}} \big(
      E(t_2 - \sigma) - E(t_1 - \sigma) \big) G(X(t_1)) \big\|^2_{L_2^0}
      \ds\Big)^{\frac{1}{2}} \Big\|_{\LOR}.
    \end{split}
  \end{align}
  By using a similar technique as for
  \eqref{eq4:T31} the first summand in \eqref{eq4:T5} is estimated by
  \begin{align*}
    &\Big\| \Big( \int_{0}^{t_1} \big\| A^{\frac{1}{2}} \big(
    E(t_2 - \sigma) - E(t_1 - \sigma) \big) \big( G(X(\sigma)) - G(X(t_1))
    \big) \big\|^2_{L_2^0}
    \ds\Big)^{\frac{1}{2}} \Big\|_{\LOR}^2\\
    &\; = \Big\| \int_{0}^{t_1} \big\| A^{\frac{1}{2}} \big(
    E(t_2 - \sigma) - E(t_1 - \sigma) \big) \big( G(X(\sigma)) - G(X(t_1))
    \big) \big\|^2_{L_2^0} \ds  \Big\|_{L^{p/2}(\Omega;\mathbb{R})}
    \\
    &\;\le C (t_2 - t_1)^{\eta} \Big\| \int_{0}^{t_1} (t_1 -
    \sigma)^{-1-\eta} 
    \left\| G(X(\sigma)) - G(X(t_1)) \right\|^2_{L_2^0} \ds
    \Big\|_{L^{p/2}(\Omega;\mathbb{R})}\\
    &\;\le C (t_2 - t_1)^{\eta} \int_{0}^{t_1} (t_1 - \sigma)^{-1 -\eta}
    \left\| X(\sigma) - X(t_1) \right\|^2_{\LOH} \ds\\
    &\;\le C (t_2 - t_1)^{\eta} \frac{1}{1 - \eta} t_1^{1 - \eta}.
  \end{align*}
  For the first inequality we applied Lemma \ref{lem:1} \emph{(i)} and
  \emph{(ii)} with an 
  arbitrary parameter $\eta \in (0,1)$. Then we used \eqref{eq1:G_lip} and the
  $\frac{1}{2}$-H\"older continuity of $X$. It follows that the summand
  vanishes in the limit $t_2 -t_1 \to 0$.
  }

  {
  For the second summand in \eqref{eq4:T5} it holds that
  \begin{align*}
    &\Big\| \Big( \int_{0}^{t_1} \big\| A^{\frac{1}{2}} \big(
    E(t_2 - \sigma) - E(t_1 - \sigma) \big) G(X(t_1)) \big\|^2_{L_2^0}
    \ds\Big)^{\frac{1}{2}} \Big\|_{\LOR}\\
    &\quad = \Big\| \Big( \sum_{m=1}^\infty \int_{0}^{t_1} \big\|
    A^{\frac{1}{2}} E(t_1 -\sigma) 
    \big( E(t_2 - t_1) - I \big) G(X(t_1)) \varphi_m \big\|^2
    \ds \Big)^{\frac{1}{2}} \Big\|_{\LOR} \\
    & \quad \le C  \big\| \big( E(t_2 - t_1) - I \big) G(X(t_1))
    \big\|_{L^p(\Omega;L_2^0)}\\
    &\quad \le C \big\|  G(X(t_1)) -
    G(X(t_3))  \big\|_{L^p(\Omega;L_2^0)} + \big\| \big( E(t_2 - t_1) - I \big)
    G(X(t_3)) \big\|_{L^p(\Omega;L_2^0)} ,
  \end{align*}
  where we used Lemma \ref{lem:1} \emph{(iii)}. By Assumption \ref{as:1} and
  \eqref{eq3:1} it holds that
  \begin{align*}
    \lim_{t_1 \to t_3} \big\|  G(X(t_1)) - G(X(t_3))
    \big\|_{L^p(\Omega;L_2^0)} = 0.
  \end{align*}
  Since, as above, Lebesgue's dominated convergence theorem yields that 
  \begin{align*}
    \lim_{t_1 \to t_3 \atop t_2 \to t_3} \big\| \big( E(t_2 - t_1) - I \big)
    G(X(t_3)) \big\|_{L^p(\Omega;L_2^0)} = 0
  \end{align*}
  the proof for $T_5$ is complete. This completes the proof of
  the theorem.
  }
\end{proof}

\begin{remark}
  If one wants to extend the result of Theorem \ref{th:3} to general $r \in
  [0,1)$ it is not hard to adapt the given arguments for all terms $T_i$, $i
  \in  \{1,2,3,5\}$. 
  
  For $T_4$, however, the situation is more delicate. This
  becomes apparent in the discussion of \eqref{eq4:T42}, which for $r \in
  (0,1)$ is equal to
  \begin{align*}
    &\Big\| \Big( \int_{t_1}^{t_2} \big\| A^{\frac{1+r}{2}} E(t_2 - \sigma)
    \big( G(X(t_2)) -
    G(X(t_3)) \big) \big\|_{L_2^0} \ds \Big)^{\frac{1}{2}}
    \Big\|^2_{\LOR}\\
    &\quad \le C \big\| A^{\frac{r}{2}} \big( G(X(t_2)) -
    G(X(t_3)) \big) \big\|_{L^p(\Omega;L_2^0)}.
  \end{align*}
  Unlike the case $r=0$ we do not want to assume that $x \mapsto
  A^{\frac{r}{2}} G(x)$ is globally Lipschitz continuous. Therefore, we cannot
  directly conclude that this term vanishes in the limit $t_2 \to t_3$. 

  However, for $p \in (2,\infty]$, it is enough to assume that the
  mapping $x \mapsto A^{\frac{r}{2}} G(x)$ is continuous. In order to
  show this we use a generalized version of Lebesgue's dominated
  convergence theorem (see \cite[1.23]{alt2006}), which allows a
  $t_2$-dependent family of dominating functions.

  In fact, by the linear growth
  condition \eqref{eq1:G_lin} we obtain, for almost all
  $\omega \in \Omega$, 
  \begin{align*}
    \big\| A^{\frac{r}{2}} \big( G(X(t_2,\omega)) -
    G(X(t_3,\omega)) \big) \big\|_{L_2^0} \le C \big( 1+ \| A^{\frac{r}{2}}
    X(t_2,\omega) \| + \|A^{\frac{r}{2}} X(t_3,\omega) \| \big),   
  \end{align*}
  where, by \eqref{eq3:1}, the family of dominating functions
  converges:
  \begin{align*}
    \big(1+ \| A^{\frac{r}{2}}
    X(t_2) \| + \|A^{\frac{r}{2}} X(t_3) \| \big) \to \big( 1+ 2
    \|A^{\frac{r}{2}} X(t_3) \| \big) \quad 
    \text{ in } L^p(\Omega,\R)
    \text{ as } t_2 \to t_3 . 
  \end{align*}
  Further, by \eqref{eq3:1} with $p \in (2,\infty)$ Kolmogorov's
  continuity theorem \cite[Th.~3.3]{daprato1992} yields that there
  exists a continuous version of the process $t \mapsto
  A^{\frac{r}{2}} X(t)$. Hence, under the additional assumption that
  $x \mapsto A^{\frac{r}{2}} G(x)$ is continuous we obtain, for almost
  all $\omega \in \Omega$,
  \begin{align*}
    \lim_{t_2 \to t_3} \big\| A^{\frac{r}{2}} \big( G(X(t_2,\omega)) -
    G(X(t_3,\omega)) \big) \big\|_{L_2^0} = 0.
  \end{align*}
  By the generalized version of the dominated convergence theorem
  (see \cite[1.23]{alt2006}) we conclude that 
  the unique mild solution $X$ to \eqref{eq1:SPDE} is continuous with respect
  to $\big(\E[ \| \cdot \|^p_{r+1} ]\big)^{\frac{1}{p}}$.
  
  The case $p=2$ remains as an open problem. 
\end{remark}

\section{Additive noise and optimal regularity}
\label{sec:add_noise}
In this section we briefly review the assumptions and our results in the
case of additive noise, that is, we consider the case where $G \in L_2^0$ is
independent of $x$. Then the SPDE \eqref{eq1:SPDE} has the form
\begin{align}
  \begin{split}
  \dd X(t) + \left[ AX(t) + F(X(t)) \right] \dt &= G \dWt,\quad
  \text{for } 0 \le t \le T,\\
  X(0) &= X_0.
\end{split}
\label{eq4:SPDE_add}
\end{align}
For related regularity results in this special case we refer to
\cite[Ch. 5]{daprato1992}. 

Since now $G$ is a fixed bounded linear operator Assumption \ref{as:1} is
simplified to 
\begin{assumption}[Additive noise]
  \label{as:4}
  The Hilbert-Schmidt operator $G$ satisfies
  \begin{align}
    \| G \|_{L_{2,r}^0} = \| A^{\frac{r}{2}} G \|_{L_2^0} < \infty.
    \label{eq4:G}
  \end{align}
\end{assumption}

Recall that the covariance operator $Q$ of the Wiener process $W$ is 
incorporated into the norm $\| \cdot \|_{L^0_2}$. If, for example, $H=U$
and $G$ is the identity $I\colon H \to H$, then
\eqref{eq4:G} reads as follows
\begin{align*}
  \| I \|_{L_{2,r}^0} = \sum_{m = 1}^\infty \big\| A^{\frac{r}{2}}
  Q^{\frac{1}{2}} \varphi_m \big\|^2 < \infty,
\end{align*}
where $(\varphi_m)_{m \ge 1}$ denotes an arbitrary orthonormal basis of the
Hilbert space $H$. This is a common assumption on the
covariance operator $Q$ (see \cite{daprato1992}). In particular, for $r=0$ this
condition becomes $\| I \|_{L_{2}^0} = \mathrm{Tr}(Q) < \infty$.

Our result for additive noise is summarized by the following
corollary.

\begin{corollary}[Additive noise]
  \label{cor:1}
  If the Assumptions \ref{as:3}, \ref{as:2} and \ref{as:4} hold for
  some $r\in[0,1]$, $p\in[2,\infty)$, then the unique mild solution $X
  \colon [0,T] \times \Omega \to H$ to \eqref{eq4:SPDE_add} takes
  values in $\dot{H}^{r+1}$.  Moreover, for every $s \in [0,r+1]$, the
  solution process is continuous with respect to $\big(\E \left[ \|
    \cdot \|_{s}^p \right] \big)^{\frac{1}{p}}$ and fulfills
  \begin{align*}
    \sup_{t_1,t_2 \in [0,T], t_1 \neq t_2} \frac{\big(\E \left[ \| X(t_1) -
    X(t_2) \|_{s}^p\right] \big)^{\frac{1}{p}} }{|t_1
    -t_2|^{\min(\frac{1}{2},\frac{r + 1 - s}{2})}} < \infty.
  \end{align*}    
\end{corollary}

We stress that the case $r=1$ is now included. In fact, the only place,
where $r<1$ is required, is the estimate \eqref{eq2:2} and its
consequences. But in the case of additive noise the left-hand side of
this estimate is equal to zero and we avoid this problem. The same is
true for the proof of continuity, where the critical terms vanish
analogously (c.f.\ the proof of Theorems \ref{prop:2} and \ref{th:3}).

We conclude this section by an example, which demonstrates that our
spatial regularity results are optimal. Without loss of generality we
restrict our discussion to the case $p=2$. For $p>2$ one may use the results
on the optimal regularity of the stochastic convolution from
\cite{neerven2010} or \cite{brzezniak2009}.

\begin{example}
  Let $H = L^2(0,1)$ be the space of all square integrable real-valued
  functions which are defined on the unit interval $(0,1)$. Further,
  assume that $-A$ is the Laplacian with Dirichlet boundary
  conditions. In this situation the orthonormal eigenbasis
  $(e_k)_{k\ge1}$ of $-A$ is explicitly known to be
  \begin{align*}
    \lambda_k = k^2 \pi^2 \quad \text{and} \quad e_k(y) = \sqrt{2} \sin(k\pi y)
    \quad \text{for all } k \ge 1, \; y \in (0,1).
  \end{align*}
  Consider the SPDE
  \begin{align}
    \begin{split}
      \dd X(t) + A X(t) \dt &= G \dWt, \quad \text{for } 0 \le t \le T,\\
      X(0) &= 0.
    \end{split}
    \label{eq5:ex1}
  \end{align}
  We choose the operator $G$ to be the identity on $H$, and $W$ to be a
  $Q$-Wiener process on $H$, 
  where the covariance operator $Q \colon H \to H$ is given by
  \begin{align*}
    Q e_1 = 0, \quad 
    Q e_k = \frac{1}{k \log(k)^2} e_k \quad \text{ for all } k
    \ge 2.
  \end{align*}
  Then we have
  \begin{align*}
    \| G \|_{L^0_{2,r}} = \sum_{k=2}^\infty \big\| A^{\frac{r}{2}} 
    Q^{\frac{1}{2}} e_k \big\|^2 = \sum_{k=2}^\infty \lambda_k^{r}
    \frac{1}{k \log(k)^2} =
    \pi^{2r} \sum_{k=2}^\infty \frac{k^{2r}}{k \log(k)^2}. 
  \end{align*}
  Since this series converges only with $r=0$, Assumption \ref{as:4} is
  satisfied only for $r = 0$. Corollary \ref{cor:1} yields
  that the 
  mild solution $X$ to \eqref{eq5:ex1} takes values in $\dot{H}^{1}$. In the
  following we show that this result cannot be improved.
  
  In our example the mild formulation \eqref{eq1:mild} reads
  \begin{align*}
    X(t) = \int_{0}^{t} E(t-\sigma) G \dWs.
  \end{align*}
  Hence, by the It\^o-isometry for the stochastic integral we have
  \begin{align*}
    \E \big[ \big\| A^{\frac{1+r}{2}} X(t) \big\|^2 \big] &=
    \int_{0}^{t} \big\| A^{\frac{1+r}{2}} E(t-\sigma) G \big\|^2_{L_2^0} \ds \\
    &= \int_{0}^{t} \sum_{k=2}^\infty \lambda_k^{1+r}
    \ee^{-2\lambda_k(t-\sigma)} \frac{1}{k \log(k)^2} \ds \\
    &= \frac{1}{2} \sum_{k=2}^\infty \lambda_k^{r} \big( 1 -
    \ee^{-2\lambda_k t} \big) \frac{1}{k \log(k)^2} \\
    &\ge
    \frac{1}{2} \pi^{2r} \big( 1 - \ee^{-2\lambda_1 t} \big) \sum_{k=2}^{\infty}
    \frac{k^{2r}}{k \log(k)^2} = \infty \quad \text{for all } t > 0,\; r >
    0. 
  \end{align*}
  Thus, $X(t) \notin L^2(\Omega; \dot{H}^{1+r})$ for $r>0$.
\end{example}

%\bibliographystyle{plain}
%\bibliography{lit}
\def\cprime{$'$} \def\polhk#1{\setbox0=\hbox{#1}{\ooalign{\hidewidth
  \lower1.5ex\hbox{`}\hidewidth\crcr\unhbox0}}}

\end{document}